\documentclass[11pt,letterpapert]{article}
\usepackage{graphicx}
\usepackage{amssymb}
\usepackage{epstopdf}
\usepackage{fancyhdr}
\usepackage{amsmath}
\usepackage{amsthm}
\usepackage{hyperref}
\usepackage{makeidx}
\usepackage{mathdesign}
\usepackage{color}
\usepackage{geometry}
\usepackage{soul}
\usepackage{nameref}

\definecolor{purple}{rgb}{.9,0,.9}

\makeatletter
\let\orgdescriptionlabel\descriptionlabel
\renewcommand*{\descriptionlabel}[1]{%
  \let\orglabel\label
  \let\label\@gobble
  \phantomsection
  \edef\@currentlabel{#1}%
  \let\label\orglabel
  \orgdescriptionlabel{#1}%
}
\makeatother

\newcommand{\Real}{\mathbb{R}}

\newcommand{\Int}{\mathbb{Z}}

\newcommand{\rfa}{\quad {\rm for \ all}\ }

\newcommand{\cD}{{\cal D}}

\newcommand{\cL}{{\cal L}}

\newcommand{\cV}{{\cal V}}\newcommand{\cW}{{\cal W}}

\newcommand{\bb}{{\bf b}}

\newcommand{\bu}{{\bf u}}
\newcommand{\bv}{{\bf v}}\newcommand{\bw}{{\bf w}}
\newcommand{\bA}{{\bf A}}
\newcommand{\bB}{{\bf B}}\newcommand{\bC}{{\bf C}}\newcommand{\bD}{{\bf D}}
\newcommand{\bE}{{\bf E}}
\newcommand{\bH}{{\bf H}}
\newcommand{\bK}{{\bf K}}\newcommand{\bL}{{\bf L}}\newcommand{\bM}{{\bf M}}

\newcommand{\bQ}{{\bf Q}}\newcommand{\bS}{{\bf S}}
\newcommand{\bT}{{\bf T}}

\newcommand{\Lin}{\mathop{\rm Lin}}

\newcommand{\Sym}{\mathop{\rm Sym}}

\newcommand{\bbC}{\mathbb{C}}
\newcommand{\bbS}{\mathbb{S}}
\newcommand{\bbA}{\mathbb{A}}
\newcommand{\bbT}{\mathbb{T}}

\newcommand{\bpsi}{\boldsymbol{\psi}}
\newcommand{\bphi}{\boldsymbol{\phi}}
\newcommand{\ve}{\varepsilon}

\newtheorem{theorem}{Theorem}[section]
\newtheorem{lemma}[theorem]{Lemma}
\newtheorem{corollary}[theorem]{Corollary}
\newtheorem{definition}[theorem]{Definition}
\newtheorem{proposition}[theorem]{Proposition}

\def\Xint#1{\mathchoice
{\XXint\displaystyle\textstyle{#1}}%
{\XXint\textstyle\scriptstyle{#1}}%
{\XXint\scriptstyle\scriptscriptstyle{#1}}%
{\XXint\scriptscriptstyle\scriptscriptstyle{#1}}%
\!\int}
\def\XXint#1#2#3{{\setbox0=\hbox{$#1{#2#3}{\int}$ }
\vcenter{\hbox{$#2#3$ }}\kern-.6\wd0}}

\def\dashint{\Xint-}

\newcommand{\beqn}{\begin{equation}}
\newcommand{\eeqn}{\end{equation}}

\newcommand{\bone}{\textbf{1}}
\newcommand{\bzero}{{\bf 0}}

\newcommand{\abs}[1]{|#1\rvert}
\newcommand{\norm}[1]{\left\|#1\right \|}

\newcommand{\dv}{\,\text{dv}}

\newcommand{\trans}{{\scriptscriptstyle\mskip-1mu\top\mskip-2mu}}

\title{On the homogenization of a new class of locally periodic microstructures in linear elasticity with residual stress}
\author{Brian Seguin}

\begin{document}
\date{}

\maketitle


\begin{abstract}
\noindent Many biological and engineering materials have nonperiodic microstructures for which classical periodic homogenization results do not apply.  Certain nonperiodic microstructures may be approximated by locally periodic microstructures for which homogenization techniques are available.  Motivated by the consideration that such materials are often anisotropic and can posses residual stresses, a broad class of locally periodic microstructures is considered and the resulting effective macroscopic equations are derived.  The effective residual stress and effective elasticity tensor are determined by solving unit cell problems at each point in the domain.  However, it is found that for a certain class of locally periodic microstructures, solving the unit cell problems at only one point in the domain completely determines the effective elasticity tensor.\\

\noindent \textbf{Keywords:} homogenization, elasticity, microstructure, residual stress\\

\noindent \textbf{AMS subject classification:} 35B27, 35Q74, 74Q15 
\end{abstract}

\section{Introduction}

A great deal of work has been done on the homogenization of materials with periodic microstructure.  See, for example, the books by Cioranescu and Donato \cite{CD}, Oleinik, Shomaev, and Yosifian \cite{OSY}, Jikov, Kozlov and Oleinik \cite{JKO}, Mei and Vernescu \cite{MV}, Bensoussan, Lions and Papanicolaou \cite{BLP}, and the references therein.  For periodic microstructure, the effective properties of the material, often called the effective or macroscopic coefficients, are determined by solving unit cell problems.  From the early years of homogenization theory, results were also known for nonperiodic microstructures.  See, for example, the works of Spagnolo \cite{Spa}, Murat \cite{Murat}, and Tartar \cite{Tartar}.  However, these results do not provide a way of computing the effective coefficients.  

One of the avenues of research in the field of homogenization is to broaden the class of microstructures in which the effective coefficients can be explicitly calculated.  This is often accomplished by considering microstructures that can be related to a periodic structure.  Depending on how the structure compares to a periodic one results in different avenues of generalization.  Microstructures that are uniformly close to a periodic structure are called almost periodic, see for example \cite{CDG2, AF}, while microstructures that can be approximated locally by periodic structures are called locally periodic.  In this work we focus on the later type of microstructure.  Briane \cite{Briane90,Briane93,Briane94} was interested in microstructures that are diffeomorphic to a periodic structure.  Using $H$-convergence , Briane found that such microstructures could be approximated by locally periodic microstructures, which consist of patches of periodic structure and the periodicity in nearby patches are closely related.  In the homogenized limit of such structures, both the size of the patches and the period of the microstructure within each patch goes to zero, but the periods go to zero more rapidly.  The resulting effective coefficients are found by solving a unit cell problems at each point in the domain.  As Shkoller \cite{Shk} points out, this is not computationally feasible, but it is possible to approximate the effective coefficients by solving only a finite number of unit cell problems.  Motivated by Briane's work, Alexandre \cite{Ale} introduced $\theta$--2 convergence, which is analogous to the notion of two-scale convergence first introduced by Nguetseng \cite{Nguetseng89} and further developed by Allaire \cite{Allaire92}, and is well-suited for smoothly transformed periodic structures.  A further generalization, under the name of scale convergence, was accomplished by Mascarenhas and Toader \cite{MT} using Young measures.  Looking back at Briane's work, Ptashnyk \cite{Plpts} considered a broader class of locally periodic microstructures which are not necessarily obtained by transforming a periodic structure with a smooth function.  To homogenize such microstructures, Ptashnyk introduced the concept of locally periodic two-scale convergence.  She \cite{PlptsUO} also developed an unfolding operator associated with this type of two-scale convergence which is analogous to the unfolding operator conceived by Cioranescu, Damlamian and Griso \cite{CDG} for classical two-scale convergence.

Motivated by the types of microstructures appearing in biomaterials that are the composite of anisotropic constituents, here a broad class of locally periodic microstructures is considered.
Many biological materials posses microstructures that are not periodic so that the classical homogenization results do not apply.  However, they often have a structure that is approximately locally periodic.  See, for example, \cite{DMKM,EM,FSTR,GDMMFT,HQG,MPR,Retal,RBF}.  A common type of microstructure in biomaterials, as well as engineering applications, is fibrous in nature in that the material can be modeled as a matrix embedded with fibers.  If the orientation of the fibers varies smoothly within the material, then the fibrous microstructure is not periodic, but can be approximated by a locally periodic microstructure.  Such structures were explicitly considered in \cite{Briane93,Plpts}.  However, when the matrix and the fibers are not isotropic, the approach taken by previous authors does not accurately reflect the fibrous microstructure.  The reason for this is that the orientation of the anisotropy of the fibers were not changed as the orientation of the fibers changed.  

A specific example of a nonperiodic microstructure in nature where this issue appears involves plant cell walls \cite{EM}, which can be modeled as an isotropic matrix, composed mainly of polysaccharides and water, embedded with cellulose microfibrils that are anisotropic \cite{DMKM}.  The anisotropy of the fibrils are aligned with the axis of the microfibril and, hence, when the orientation of the fibrils changes, the orientation of the anisotropies changes as well.

In this work, using Ptashnyk's \cite{Plpts} locally periodic two-scale convergence, we consider the homogenized limit of a material with a locally periodic microstructure in which the anisotropy in each patch of the microstructure is transformed.  This is done in the context of linear elasticity.  First we establish the result in the case where the anisotropy and the periodicity are transformed independently.  However, the case in which the anisotropy and periodicity change according to the same transformation is discussed in detail.  It is found that in this situation the resulting homogenized material is materially uniform but, in general, inhomogeneous.\footnote{For a detailed discussion of materially uniformity and inhomogeneity, see Noll \cite{N67}, Epstein and El$\dot{\text{z}}$anowski \cite{EE}, or Wang \cite{W67}.}  It is also found that in this case only six unit cell problems are needed to completely determine the effective elasticity tensor.  This homogenized tensor is still spatially dependent, but its values at different points are related to each other through the transformation that describes how the orientation of the microstructure changes from point to point.  A sketch of how to homogenize a material with a nonperiodic microstructure by approximating it with a material that has a locally locally periodic microstructure is also presented.

The transformation of the periodicity within each patch of a locally periodic microstructure is achieved by transforming a representative unit cell.  When the anisotropy is transformed as well, this is akin to changing the reference configuration for the constitutive law.  When the new reference configuration is related to the original one through a transformation that is not orthogonal, residual stresses appear.  The analysis below allows for such affects.  In the homogenized limit, a nonconstant effective residual stress term is present.  Perhaps surprisingly, a unit cell problem must be solved to determine this stress.  In contrast to the effective elasticity tensor, even in the case when the anisotropy and the periodicity are changed by the same transformation, a unit cell problem must be solved at each point in the domain to determine the effective residual stress.

In Section \ref{sectlpmc} we motivate the form of the constitutive law appearing in the equation we homogenize and give the precise definition of a locally periodic microstructure.  Section \ref{sectlptsc} recalls Ptashnyk's \cite{Plpts} locally periodic two-scale convergence and lists the results that will be used later.  The main homogenization result is established in Section \ref{sectHom}.  Finally, in Section \ref{sectInterp}, the results are discussed and several special cases are considered.




\section{Locally periodic microstructure}\label{sectlpmc}

To motivate the types of locally periodic microstructures considered here, we begin by discussing an elastic constitutive law relative to a reference configuration that is stress free.  This law is then formulated relative to a new reference configuration that is obtained from the original by a transformation that is close to orthogonal.  The resulting constitutive law is then linearized.  

We work in $\Real^n$ with arbitrary dimension $n$, however the most physically interesting cases are $n=2,3$.  Let $\Lin$ denote the set of linear mappings from $\Real^n$ to $\Real^n$ and $\Sym$ those elements of $\Lin$ that are symmetric.  Since fourth-order tensors can be viewed as linear mappings from $\Lin$ to $\Lin$, the space of all fourth-order tensors will be denoted by $\Lin(\Lin)$.

Consider an elastic material with reference configuration $Y$, where $Y$ is a parallelepiped in $\Real^n$, whose constitutive law is
\beqn\label{originalconst}
\bS(y)=\hat\bS(\bC,y)\rfa y\in Y,
\eeqn
where $\bC=(\bone+\nabla\bu)^\trans(\bone+\nabla\bu)$ is the right Cauchy--Green strain tensor, $\bu$ is the displacement, and $\bS$ is the second Piola stress.  Assume that $\hat \bS(\bone,y)=\textbf{0}$ for all $y\in Y$ so that the reference configuration $Y$ is stress free.  Let $h(y)=\bH y+\tilde x$ be a transformation of $Y$, with $\bH\in\Lin$ invertible and $\tilde x\in\Real^n$ fixed.  The constitutive law \eqref{originalconst} relative to $h(Y)$ is given by\footnote{See, for example, Gurtin, Fried and Anand \cite{GFA}.}
\beqn
\bS(x)=\bH\hat\bS(\bH^\trans\bC\bH,\bH^{-1}(x-\tilde x))\bH^\trans\rfa x\in h(Y).
\eeqn
where $\bC$ is the right Cauchy--Green tensor relative to the new configuration.  Given a small positive $\ve$, assume that the displacement gradient is of order $\ve$ and that $\bH$ differs from an orthogonal linear mapping by a term of order $1$ so that
\begin{align}
\bC&=\bone + \nabla\bu+\nabla\bu^\trans + o(\ve),\\
\bone&=\bH^\trans\bH + o(1).
\end{align}
Setting
\beqn
\bbC(y)=2\nabla_\bC\hat\bS(\bone,y)\rfa y\in Y,
\eeqn
for any $x\in h(Y)$ we have
\begin{align}
\bS(x) &= \bH\hat\bS(\bH^\trans\bH,\bH^{-1}(x-\tilde x))\bH^\trans+\bH\nabla_\bC\hat\bS(\bH^\trans\bH,\bH^{-1}(x-\tilde x))[\bH^\trans(\bC-\bone)\bH]\bH^\trans+o(\ve)\\
& =  \bH\hat\bS(\bH^\trans\bH,\bH^{-1}(x-\tilde x))\bH^\trans+\bH\bbC(\bH^{-1}(x-\tilde x))[\bH^\trans\bE\bH]\bH^\trans+o(\ve),
\end{align}
where $\bE=\frac{1}{2}(\nabla\bu+\nabla\bu^\trans)$.  Thus, the linearized constitutive law reads
\beqn\label{linconst}
\bS(x) = \bH\hat\bS(\bH^\trans\bH,\bH^{-1}(x-\tilde x))\bH^\trans+\bH\bbC(\bH^{-1}(x-\tilde x))[\bH^\trans\bE\bH]\bH^\trans.
\eeqn
The second term on the right-hand side \eqref{linconst} depends linearly on the strain $\bE$ and is the stress due to the deformation relative to the configuration $h(Y)$.  The first term on the right-hand side remains when the strain is zero and, hence, is a residual stress.  In the case where $\bH$ is orthogonal, the residual stress vanishes.

While $\hat\bS$ is only defined on $Y$, it can be extend to all of $\Real^n$ so that it is $Y$-periodic.  When this is done, \eqref{linconst} is $h(Y)$-periodic. We want to consider a material whose microstructure consists of patches of periodic structure such that nearby patches are similar.  For example, one patch could have periodic structure of the form \eqref{linconst} while a nearby patch will have structure of the form \eqref{linconst} with $\bH$ replaced by a transformation that is close to $\bH$.  This is the idea behind a locally periodic microstructure.

To make this idea precise, begin by considering an open subset $\Omega$ of $\Real^n$ that is bounded with Lipschitz boundary, and assume that $Y$ has unit volume.  The set $\Omega$ will be the elastic material with locally periodic microstructure.  Consider a fixed $\ve>0$ (no relation to the $\ve$ used for the linearization) and $r\in (0,1)$.  For $k\in\Int^n$, set $\Omega^\ve_k=(0,\ve^r)^n+\ve^rk$.  Let $I^\ve$ be the unique subset of $\Int^n$ such that
\beqn
\Omega_k^\ve\cap\Omega\not=\emptyset\ \ \text{for all }\ k\in I^\ve\qquad\text{and}\qquad \Omega\subset\bigcup_{k\in I^\ve}\overline\Omega_k^\ve.
\eeqn
Each $\Omega_k^\ve$, for $k\in I^\ve$, is a patch of $\Omega$ in which the microstructure will be periodic.  While the size of $\Omega_k^\ve$ is of order $\ve^r$, the periodicity within $\Omega_k^\ve$ will be on the order of $\ve$.

Consider a function $\bH:\Real^n\longrightarrow\Lin$ with $\bH(x)$ invertible for all $x\in\Real^n$.  We shall use the notation $\bH_x=\bH(x)$ and $Y_x=\bH_x(Y)$ for $x\in\Real^n$.  The function $\bH$ will describe how the unit cell $Y$ is transformed so as to define the periodic structure within each patch $\Omega_k^\ve$.   


To specify the microstructure, for each $k\in I^\ve$ choose $x_k^\ve,\tilde x_k^\ve\in \Omega_k^\ve$ arbitrarily.  Motivated by \eqref{linconst}, define the stress in $\Omega_k^\ve$ by
\beqn\label{patchconst}
\bS(x)=\hat\bS_r(\bH_{x^\ve_k},\bH_{x^\ve_k}^{-1}(x-\tilde x_k^\ve)/\ve) + \hat\bbS(\bbC(\bH_{x^\ve_k}^{-1}(x-\tilde x_k^\ve)/\ve),\bH_{x^\ve_k})\bE\rfa x\in\Omega_k^\ve
\eeqn
where for any $\bbA\in\Lin(\Lin)$, $\bA,\bE\in\Lin$, and $y\in Y$,
\begin{align}
\label{Ssp}\hat\bS_r(\bA,y)&=\bA\hat\bS(\bA^\trans\bA,y)\bA^\trans,\\
\label{bbSsp}\hat\bbS(\bbA,\bA)\bE&=\bA(\bbA[\bA^\trans\bE\bA])\bA^\trans.
\end{align}
Notice that the stress defined in \eqref{patchconst} is $\ve Y_{x_k^\ve}$-periodic since $\hat \bS$ and $\bbC$ are $Y$-periodic.  Set
\beqn\label{YStranseq}
\bS_r^\ve(x)=\sum_{k\in I^\ve} \hat\bS_r(\bH_{x^\ve_k},\bH_{x^\ve_k}^{-1}(x-\tilde x_k^\ve)/\ve)\chi_{\Omega_k^\ve}(x)
\eeqn
and
\beqn\label{YCtranseq}
\bbC^\ve(x)=\sum_{k\in I^\ve} \hat\bbS(\bbC(\bH_{x^\ve_k}^{-1}(x-\tilde x_k^\ve)/\ve),\bH_{x^\ve_k})\chi_{\Omega_k^\ve}(x).
\eeqn
With these definitions the stress at any $x\in \Omega$ is given by
\beqn\label{Fconst}
\bS(x)=\hat\bS^\ve_r(x)+\bbC^\ve(x)\bE.
\eeqn
This equation defines a constitutive law for an elastic material that has patches of periodic microstructure.  Figure~\ref{locperMC} depicts the domain $\Omega$ broken into patches of order $\ve^r$ that have a periodic structure of order $\ve$. 

\begin{figure}
\centering
\includegraphics[width=2.8in]{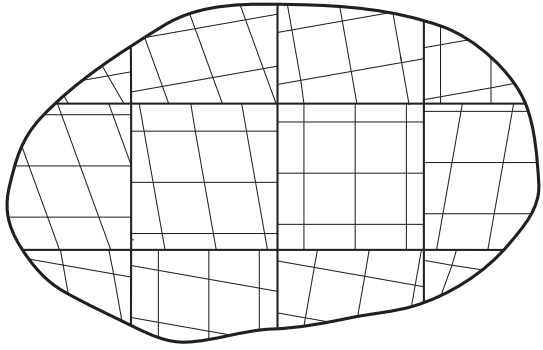}
\thicklines
\caption{A depiction of a domain $\Omega$ with locally periodic microstructure.  This domain $\Omega$ is divided into $12$ patches, each of which have slightly different periodic structure.}
\label{locperMC}
\end{figure}

Consider the displacement problem in elastostatics:
\beqn\label{dispbvp}
\left \{
\begin{array}{ll}
\text{div}(\bS_r^\ve+\bbC^\ve\nabla\bu^\ve)+\bb=\textbf{0}\qquad &\text{in}\ \Omega, \\[10pt]
\bu^\ve = \bu_\circ\qquad &\text{on}\ \partial\Omega,
\end{array}
\right.
\eeqn
where $\bu_\circ$ is a given boundary displacement and $\bb$ is a body force.  It is known that in the limit as $\ve$ goes to zero, materials with locally periodic microstructure approximate materials with nonperiodic microstructure that are common in biological materials such as muscles \cite{GDMMFT} and plant tissues \cite{FSTR}.  Such connections were made by Briane \cite{Briane94} and Ptashnyk \cite{Plpts}.  


The locally periodic microstructure described above is different from those considered before in the works for Briane\footnote{Briane did not consider an elastic microstructure.  Rather, he was looking at the diffusion equation.  However, the main difference between the microstructure described here and that considered by Briane still holds.} \cite{Briane94} and Ptashnyk \cite{Plpts}.  In these works the stress is given by \eqref{Fconst} with $\hat \bS_r^\ve=\textbf{0}$ and $\bbC^\ve$ is given by
\beqn
\bbC^\ve(x)=\sum_{k\in I^\ve} \hat\bbS(\bbC(\bH_{x^\ve_k}^{-1}(x-\tilde x_k^\ve)/\ve),\textbf{1})\chi_{\Omega_k^\ve}(x).
\eeqn
This can be interpreted as saying that the unit cell $Y$ is transformed when the microstructure in each patch is considered but the alignment of the anisotropy of $\bbC$ is not changed in each patch.  This differs from the stress determined by \eqref{YStranseq} and \eqref{YCtranseq} in which in each patch of the material the unit cell and the anisotropy are transformed according to $\bH$.  

This suggests a general class of locally periodic microstructures whose constitutive law is of the form \eqref{Fconst} with
\begin{align}
\label{YStrans}\bS_r^\ve(x)&=\sum_{k\in I^\ve} \hat\bS_r(\bK_{x^\ve_k},\bH_{x^\ve_k}^{-1}(x-\tilde x_k^\ve)/\ve)\chi_{\Omega_k^\ve}(x),\\
\label{YCtrans}\bbC^\ve(x)&=\sum_{k\in I^\ve} \hat\bbS(\bbC(\bH_{x^\ve_k}^{-1}(x-\tilde x_k^\ve)/\ve),\bK_{x^\ve_k})\chi_{\Omega_k^\ve}(x),
\end{align}
where $\bK:\Omega\rightarrow\Lin$.  In this locally periodic microstructure, $\bH$ determines how the unit cell is transformed in the different patches within $\Omega$ and $\bK$ determines how the anisotropy of $\bbC$ is transformed in the different patches and determines if there is any residual stress.  When $\bH=\bK$, this reduces to the microstructure initially considered above and when $\bK=\textbf{1}$ this becomes the microstructure considered by Briane \cite{Briane94,Briane93} and Ptashnyk \cite{Plpts}.  However, the case $\bH\not =\bK$ as well as the case $\bH=\textbf{1}$ and $\bK\not=\textbf{1}$ describe novel microstructures.  The different effects associated with the different choices of $\bH$ and $\bK$ on the unit cell $Y$ are depicted in Figure~\ref{figYCtrans}.

The precise assumptions necessary to homogenize \eqref{dispbvp} and the derivation of the homogenized equation appear in Section~\ref{sectHom}.  To take the limit as $\ve$ goes to zero in \eqref{dispbvp} we use locally periodic two-scale convergence, which is described in the next section.

\begin{figure}
\centering
\includegraphics[height=1in]{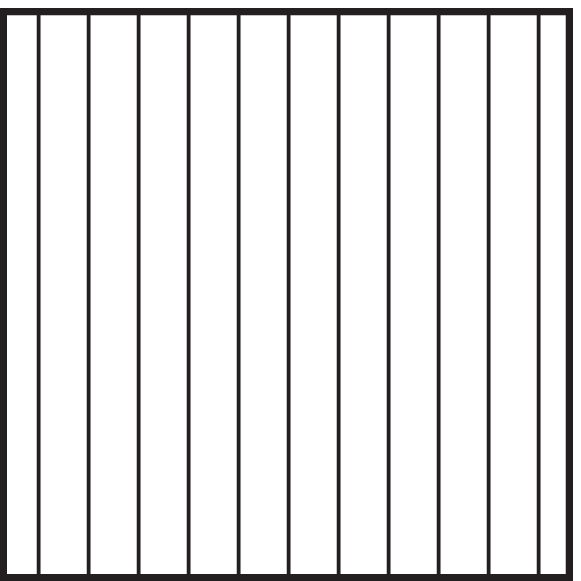}
\hspace{.5in}
\includegraphics[height=1in]{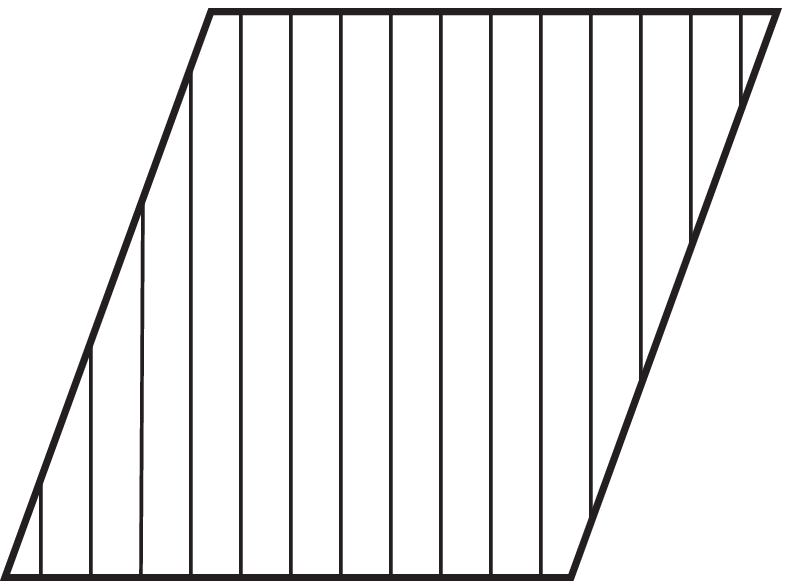}
\hspace{.5in}
\includegraphics[height=1in]{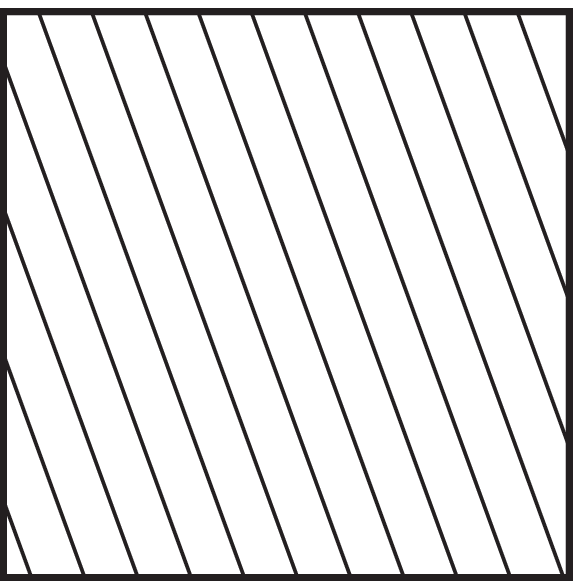}
\par\vspace{.1in}
\includegraphics[height=1in]{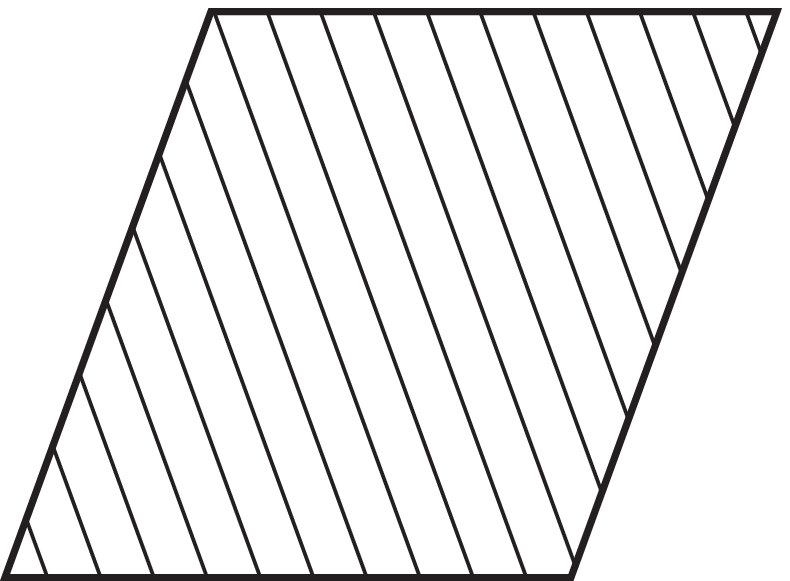}
\hspace{.5in}
\includegraphics[height=1.1in]{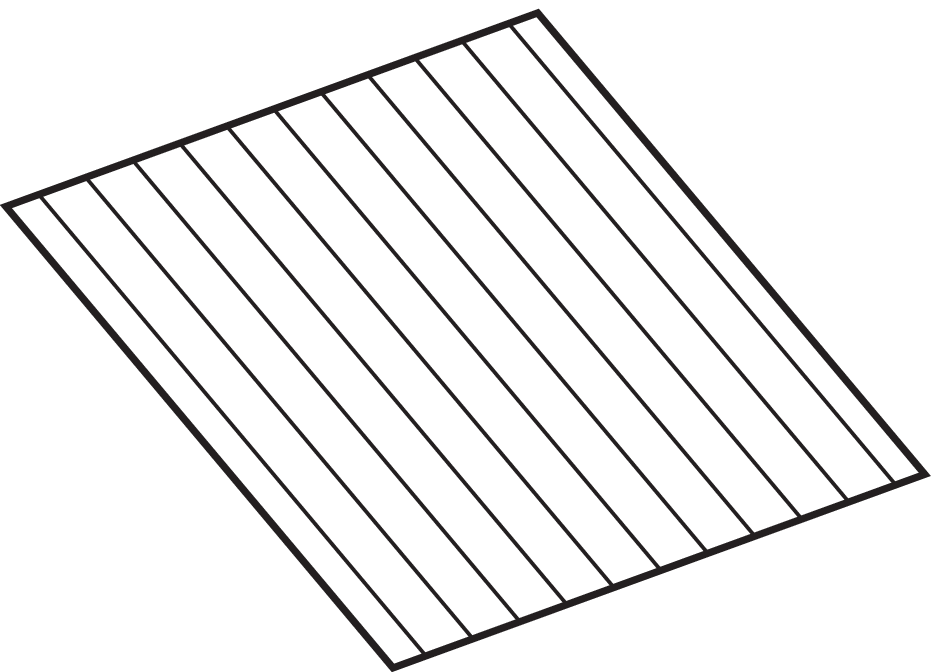}
\thicklines
\put(-262,90){(a)}
\put(-135,90){(b)}
\put(-5,90){(c)}
\put(-220,-15){(d)}
\put(-60,-15){(e)}
\caption{A depiction of the unit cell $Y$ and various transformations of it.  The parallel lines within each unit cell represent the orientation of the anisotropy of the elastic properties of the unit cell.  (a) A depiction of the original, untransformed unit cell $Y$.  (b) A depiction of a transformed unit cell with the elasticity properties untransformed.  This corresponds to the case where $\bK=\bone$.  (c) A depiction of an untransformed unit cell in which the elastic properties are transformed.  This results when $\bH=\bone$.  (d) A depiction of a transformed unit cell with transformed elastic properties in which the two transformations differ.  This occurs for generic $\bH$ and $\bK$.  (e) A depiction of a transformed unit cell with transformed elastic properties in which the two transformations are the same---that is, $\bH=\bK$.}
\label{figYCtrans}
\end{figure}

\section{Locally periodic two-scale convergence}\label{sectlptsc}

Unless stated otherwise, the results in this section are taken from the work of Ptashnyk \cite{Plpts} and are stated here for easy reference.  Let $\Omega$ be a bounded Lipschitz domain and $Y$ a parallelepiped in $\Real^n$.  Fix a finite-dimensional inner-product space $\cV$.  The choices $\cV$ equals $\Real^n$, $\cV=\Lin$, and $\Lin(\Lin)$ will be utilized.  Consider $C(\overline\Omega,C_\text{per}(Y,\cV))$, the set of all continuous functions $\bpsi:\overline\Omega\times \Real^n\longrightarrow\cV$ that are $Y$-periodic in their second argument.  Fix a $C^1$ function $\bH:\Real^n\longrightarrow\Lin$ with $\bH(x)$ invertible for all $x\in\Real^n$.  We shall use the notation $\bH_x=\bH(x)$ and $Y_x=\bH_x(Y)$ for $x\in\Real^n$.  For $\bpsi\in C(\overline\Omega,C_\text{per}(Y,\cV))$, define
\beqn\label{Hnotation}
\bpsi_\bH(x,y)=\bpsi(x,\bH^{-1}_xy)\rfa (x,y)\in\overline\Omega\times\Real^n.
\eeqn
This notation will also be used even if $\bpsi$ is independent of $x$.  Notice that for each $x\in\Omega$, $\psi_\bH(x,\cdot)$ is $Y_x$-periodic.  The space of all mappings of the form $\bpsi_\bH$ for $\bpsi\in C(\overline\Omega,C_\text{per}(Y,\cV))$ is denoted by $C(\overline\Omega,C_\text{per}(Y_\bH,\cV))$.  Conversely, given $\bphi\in C(\overline\Omega,C_\text{per}(Y_\bH,\cV))$, the corresponding mapping in $C(\overline\Omega,C_\text{per}(Y,\cV))$ will be denoted by $\bphi_{\bH^{-1}}$.  The spaces $L^2(\Omega,C_\text{per}(Y_\bH,\cV))$, $L^2(\Omega,L^2_\text{per}(Y_\bH,\cV))$, $C(\overline\Omega,L^2_\text{per}(Y_\bH,\cV))$, etc.~are defined in a similar way and are Banach spaces with the expected norm.  For example,
\beqn
\|\bphi\|_{L^2(\Omega,C_\text{per}(Y_\bH,\cV))}=\Big(\int_\Omega \sup_{y\in Y_x}|\bphi(x,y)|^2\, dx\Big)^{1/2}.
\eeqn

Consider $\bpsi\in C(\overline\Omega,C_\text{per}(Y_\bH,\cV))$.  The locally periodic approximation $\cL^\ve\bpsi$ of $\bpsi$ is defined as
\beqn\label{cLve}
(\cL^\ve\bpsi)(x)=\sum_{k\in I^\ve}\bpsi_{\bH^{-1}}(x,\bH^{-1}_{x^\ve_k}(x-\tilde x_k^\ve)/\ve)\chi_{\Omega^\ve_k}(x)\rfa x\in\Omega,
\eeqn
where, recall, $x_k^\ve$ and $\tilde x_k^\ve$ are arbitrary points in $\Omega_k^\ve$.  We also require the approximation
\beqn\label{cLve0}
(\cL^\ve_0\bpsi)(x)=\sum_{k\in I^\ve}\bpsi_{\bH^{-1}}(x^\ve_k,\bH^{-1}_{x^\ve_k}(x-\tilde x_k^\ve)/\ve)\chi_{\Omega^\ve_k}(x)\rfa x\in\Omega.
\eeqn
Notice that both of the approximations $\cL^\ve\bpsi$ and $\cL^\ve_0\bpsi$ are, in general, discontinuous.  To introduce a continuous approximation, fix $\rho\in(r,1)$ and for each $k\in I^\ve$ let $\phi_{\Omega^\ve_k}\in C^\infty_c(\Omega,\Real)$ be an approximation of $\chi_{\Omega^\ve_k}$ such that
\begin{align}
\label{phi1}\Big\|\sum_{k\in I^\ve} \abs{\phi_{\Omega^\ve_k}-\chi_{\Omega^\ve_k}}\Big \|_{L^2(\Real^n)}\rightarrow 0\qquad \text{as}\ \ve\rightarrow 0\\
\label{phi2}\norm{\nabla\phi_{\Omega^\ve_k}}_{L^\infty(\Real^n)}\leq C\ve^{-\rho}.
\end{align}
For the construction of $\phi_{\Omega^\ve_k}$ see Briane \cite{Briane94}.  Define the smooth approximation of $\bpsi$ by
\beqn\label{cLverho}
(\cL^\ve_\rho\bpsi)(x)=\sum_{k\in I^\ve} \bpsi_{\bH^{-1}}(x,\bH^{-1}_{x^\ve_k}(x-\tilde x_k^\ve)/\ve)\phi_{\Omega^\ve_k}(x)\rfa x\in\Omega.
\eeqn
The approximations $\cL^\ve\bpsi$, $\cL^\ve_0\bpsi$, and $\cL^\ve_\rho\bpsi$ can also be defined for $\bpsi$ in $L^2(\Omega,C_\text{\rm per}(Y_\bH,\cV))$ or $C(\overline\Omega,L^2_\text{\rm per}(Y_\bH,\cV))$.

The approximations satisfy the following convergence results.

\begin{proposition}\label{clconv}
For $\bpsi\in L^p(\Omega,C_\text{\rm per}(Y_\bH,\cV))\cup C(\overline\Omega,L^p_\text{\rm per}(Y_\bH,\cV))$, $1\leq p<\infty$, we have
\begin{align}
\lim_{\ve\rightarrow 0} \int_\Omega |(\cL^\ve\bpsi)(x)|^p\, dx&=\int_\Omega\dashint_{Y_x}|\bpsi(x,y)|^p\, dydx\\
\lim_{\ve\rightarrow 0} \int_\Omega |(\cL^\ve_0\bpsi)(x)|^p\, dx&=\int_\Omega\dashint_{Y_x}|\bpsi(x,y)|^p\, dydx.
\end{align}
and for $\bpsi\in L^1(\Omega,C_\text{\rm per}(Y_\bH,\cV))\cup C(\overline\Omega,L^1_\text{\rm per}(Y_\bH,\cV))$ we have
\begin{align}
\lim_{\ve\rightarrow 0} \int_\Omega (\cL^\ve\bpsi)(x)\, dx&=\int_\Omega\dashint_{Y_x}\bpsi(x,y)\, dydx\\
\lim_{\ve\rightarrow 0} \int_\Omega (\cL^\ve_0\bpsi)(x)\, dx&=\int_\Omega\dashint_{Y_x}\bpsi(x,y)\, dydx.
\end{align}
\end{proposition}

\noindent  The symbol $\dashint$ denotes the average integral. Notice that in the above proposition, $\bpsi$ is continuous in at least one of its arguments.  
\begin{definition}\label{lptsdef}
A family of functions $\bu^\ve\in L^2(\Omega,\cV)$ converges locally periodic two-scale to $\bu\in L^2(\Omega,L^2(Y_\bH,\cV))$ as $\ve\rightarrow 0$ if for all $\bpsi\in L^2(\Omega,C_\text{\rm per}(Y_\bH,\cV))$
\beqn\label{lpts}
\lim_{\ve\rightarrow 0}\int_\Omega \bu^\ve(x)\cdot(\cL^\ve\bpsi)(x)\, dx = \int_\Omega\int_{Y_x}\bu(x,y)\cdot\bpsi(x,y)\, dydx.
\eeqn
\end{definition}
\noindent Using a density argument, it can be shown that $W^{1,\infty}_0(\Omega,C^\infty_{\rm per}(Y_\bH,\cV))$ can be used as the space of test functions in the above definition instead of $L^2(\Omega,C_\text{\rm per}(Y_\bH,\cV))$.

The following results are analogous to the compactness results for the classical two-scale convergence of Nguetseng \cite{Nguetseng89} and Allaire \cite{Allaire92}.

\begin{theorem}
If $\bu^\ve$ is a bounded sequence in $L^2(\Omega,\cV)$, then there is a subsequence, also denoted by $\bu^\ve$, and a $\bu\in L^2(\Omega,L^2_{\rm per}(Y_\bH,\cV))$ such that 
$$\bu^\ve\rightarrow \bu\qquad \text{locally periodic two-scale as}\ \ \ve\rightarrow 0.$$
\end{theorem}

Let $\cW_\text{per}(Y,\cV)$ denote the set of equivalence classes of $H^1_\text{per}(Y,\cV)$ where two functions are equivalent if they differ by a constant vector in $\cV$.  

\begin{theorem}\label{ltscompact}
If $\bu^\ve$ is a bounded sequence in $H^1(\Omega,\cV)$ that converges weakly to $\bu\in H^1(\Omega,\cV)$, then the sequence converges locally periodic two-scale to $\bu$ and there is a subsequence of $\nabla \bu^\ve$, also denoted by $\nabla\bu^\ve$, and a $\hat\bu\in L^2(\Omega,\cW_{\rm per}(Y_\bH,\cV))$ such that
$$\nabla\bu^\ve\rightarrow \nabla \bu+\nabla_y\hat\bu\qquad \text{ locally periodic two-scale as}\ \ \ve\rightarrow 0.$$
\end{theorem}

The next result concerns the product of two sequences that converge locally periodic two-scale.  As with the previous compactness results, this one has an analog for two-scale convergence.

\begin{lemma}\label{lemstrong}
If $\bu^\ve$ is a sequence in $L^2(\Omega,\cV)$ that converges locally periodic two-scale to $\bu\in L^2(\Omega,L^2(Y_\bH,\cV))$ such that
\beqn\label{lemstrongcond}
\lim_{\ve\rightarrow 0}\norm{\bu^\ve}^2_{L^2(\Omega,\cV)} = \int_\Omega\dashint_{Y_x}\abs{\bu(x,y)}^2\,dydx,
\eeqn
then for any sequence $\bv^\ve$ in $L^2(\Omega,\cV)$ that converges locally periodic two-scale to $\bv\in L^2(\Omega,L^2(Y_\bH,\cV))$ we have
$$\bu^\ve\bv^\ve\rightarrow\dashint_{Y_x}\bu(\cdot,y)\bv(\cdot,y)\,dy\qquad \text{in }\cD'(\Omega)\ \text{as}\ \ve\rightarrow 0.$$
\end{lemma}

The following results, while not in Ptashnyk \cite{Plpts}, simplify later arguments.

\begin{proposition}\label{obtsc}
For $\bphi\in L^2(\Omega,C_\text{\rm per}(Y_\bH,\cV))\cup C(\overline\Omega,L^2_\text{\rm per}(Y_\bH,\cV))$, we have
\begin{align}
\label{Lconv}\cL^\ve\bphi\rightarrow \bphi\quad \text{locally periodic two-scale as} \quad \ve\rightarrow 0.
\end{align}
\end{proposition}
\begin{proof}
For any $\bpsi\in W^{1,\infty}_0(\Omega,C^\infty_{\rm per}(Y_\bH,\cV))$, by Proposition \ref{clconv} we have
\beqn
\lim_{\ve\rightarrow 0}\int_\Omega(\cL^\ve\bphi)\cdot(\cL^\ve\bpsi)(x)\,dx=\lim_{\ve\rightarrow 0}\int_\Omega\cL^\ve(\bphi\cdot\bpsi)(x)\,dx=\int_\Omega\dashint_{Y_x}\bphi(x,y)\cdot\bpsi(x,y)\,dydx.
\eeqn
Thus, by Definition~\ref{lptsdef}, \eqref{Lconv} holds. 
\end{proof}

\begin{proposition}\label{propneeded}
If $\bv^\ve$ converges locally periodic two-scale, then \eqref{lpts} holds for all $\bpsi\in C(\Omega,L^\infty_\text{\rm per}(Y_\bH,\cV))$.
\end{proposition}
\begin{proof}
For any $\bpsi\in C(\Omega,L^\infty_\text{\rm per}(Y_\bH,\cV))$, by Proposition \ref{obtsc}, $\cL^\ve\bpsi$ locally periodic two-scale converges to $\bpsi$.  Moreover, by Proposition \ref{clconv}, \eqref{lemstrongcond} holds with $\bu^\ve$ replaced by $\cL^\ve\bpsi$.  Thus we can apply Lemma \ref{lemstrong} to obtain
\beqn
\lim_{\ve\rightarrow 0}\int_\Omega\bv^\ve(x)\cdot(\cL^\ve\bpsi)(x)\, dx = \int_\Omega\dashint_{Y_x}\bv(x,y)\cdot\bpsi(x,y)\,dydx.
\eeqn
\end{proof}

\begin{proposition}\label{LL0conv}
For all $\bpsi\in C(\overline\Omega,L^\infty_\text{\rm per}(Y_\bH,\cV))$, we have
\beqn
\lim_{\ve\rightarrow 0}\|\cL^\ve \bpsi-\cL^\ve_0\bpsi\|_{L^2(\Omega)}=0.
\eeqn
\end{proposition}
\begin{proof}
By the definition of $\cL^\ve$ and $\cL^\ve_0$,
\beqn
\|\cL^\ve \bpsi-\cL^\ve_0\bpsi\|_{L^2(\Omega)}^2=\sum_{k\in I^\ve}\int_{\Omega_k^\ve\cap\Omega}|\bpsi_{\bH^{-1}}(x,\bH^{-1}_{x_k^\ve}(x-\tilde x_k^\ve)/\ve)-\bpsi_{\bH^{-1}}(x_k^\ve,\bH^{-1}_{x_k^\ve}(x-\tilde x_k^\ve)/\ve)|^2\, dx.
\eeqn
Since $\bpsi_{\bH^{-1}}\in C(\overline\Omega,L^\infty_\text{\rm per}(Y,\cV))$ and $\overline\Omega$ is compact, the family of functions $x\mapsto\bpsi_{\bH^{-1}}(x,y)$, indexed by $y\in Y$, is equi-uniformly continuous.  Thus, given $\xi>0$ there is an $\delta>0$ such that 
\beqn
\text{if }|x_1-x_2|<\delta,\text{ then } |\bpsi_{\bH^{-1}}(x_1,y)-\bpsi_{\bH^{-1}}(x_2,y)|<\sqrt{\xi}\qquad\text{for a.e. }y\in Y.
\eeqn
Consider $\ve$ small enough so that if $x_1,x_2\in\Omega_k^\ve$, then $|x_1-x_2|<\delta$.  It follows that for such $\ve$,
\beqn
\|\cL^\ve \bpsi-\cL^\ve_0\bpsi\|_{L^2(\Omega)}^2\leq \sum_{k\in I^\ve}\int_{\Omega_k^\ve\cap\Omega} \xi\, dx=\xi |\Omega|.
\eeqn
This proves the desired limit.
\end{proof}

\noindent Combining Propositions \ref{propneeded} and \ref{LL0conv} we obtain the next result.

\begin{corollary}\label{corconv}
If $\bv^\ve$ converges locally periodic two-scale and $\|\bv^\ve\|_{L^2(\Omega)}$ is bounded uniformly in $\ve$, then for all $\bpsi\in C(\Omega,L^\infty_\text{\rm per}(Y_\bH,\cV))$
\beqn
\lim_{\ve\rightarrow 0}\int_\Omega \bv^\ve(x)\cdot(\cL^\ve_0\bpsi)(x)\, dx = \int_\Omega\int_{Y_x}\bv(x,y)\cdot\bpsi(x,y)\, dydx.
\eeqn
\end{corollary}

\section{Homogenization result}\label{sectHom}

In this section we formulate the precise assumptions necessary to homogenize the elasticity equation \eqref{dispbvp}.  The resulting macroscipic equation involves a macroscopic residual stress and elasticity tensor which are determined by solving unit cell probelms.

Let $\Omega$ be an open, bounded, Lipschitz domain and $Y$ a parallelepiped with unit volume.  We make the following assumptions.\\

\vspace{-.1in}
\noindent \textbf{Assumptions}\vspace{-.1in}
\begin{enumerate}
\item $\bH,\bK\in C^1(\overline\Omega,\Lin)$ and $\bH(x)$ is invertible for all $x\in\overline\Omega$\vspace{-.1in}
\item $\hat\bS_r\in C^1(\Lin,L^\infty_\text{per}(Y,\Lin))$\vspace{-.1in}
\item $\hat\bbS\in C^1(\Lin(\Lin)\times\Lin,\Lin(\Lin))$\vspace{-.1in}
\item $\bbC\in L^\infty_\text{per}(Y,\Lin(\Lin))$\vspace{-.1in}
\item if $\bbT\in\Lin(\Lin)$ is either $\bbC(y)$ or $\hat\bbS(\bbC(y),\bK_x)$ for a.e.~$y\in Y$ and $x\in\Omega$, then $\bbT$ satisfies the following conditions:\vspace{-.1in}
\begin{enumerate}
\item there is a $\alpha>0$ (independent of $y$ and $x$) such that $\alpha|\bE|^2\leq\bE\cdot\bbT\bE$ for all $\bE\in \Sym$
\item $\bA\cdot\bbT\bB = \bA\cdot\bbT\big[\frac{1}{2}(\bB + \bB^\trans)\big]=\big[\frac{1}{2}(\bA + \bA^\trans)\big]\cdot\bbT\bB$ for all $\bA,\bB\in\Lin$
\item $\bD\cdot\bbT\bE=\bE\cdot\bbT\bD$ for all $\bD,\bE\in\Sym$\vspace{-.1in}
\end{enumerate}
\item $\bb\in L^2(\Omega,\Real^n)$ and $\bu_\circ\in L^2(\partial\Omega,\Real^n)$
\end{enumerate}

\noindent Notice it is not assumed that $\hat\bS_r$ and $\hat\bbS$ are of the form \eqref{Ssp} and \eqref{bbSsp}, respectively.  The following result follows from the Lax--Milgrim theorem using standard arguments and, hence, the proof is omitted.

\begin{proposition}\label{euprop}
For every $\ve>0$, define $\bS^\ve_r$ and $\bbC^\ve$ according to \eqref{YStrans} and \eqref{YCtrans}, respectively.  Under Assumptions 1--6, the displacement problem \eqref{dispbvp} has a unique solution in $H^1(\Omega,\Real^n)$.  Moreover, the solution $\bu^\ve$ is bounded in $H^1(\Omega,\Real^n)$ independent of $\ve$.
\end{proposition}

For the remainder of the paper, $\bar y$ will denote a typical element of $Y$ while $y$ will denote a typical element of $Y_x=\bH_x(Y)$ where $x\in\Omega$.  With this convention, $\nabla_{\bar y}$ denotes the gradient with respect to a variable in $Y$ and $\nabla_y$ denotes the gradient with respect to a variable in $Y_x$.  We are now in a position to state the main result.

\begin{theorem}\label{thmmain}
Under Assumptions 1--6, the solutions $\bu^\ve$ of \eqref{dispbvp} converge weakly in $H^1(\Omega,\Real^n)$ to a function $\bu$ which is the unique solution of
\beqn\label{dispbvphom}
\left \{
\begin{array}{ll}
\text{\rm div}(\bS_{r\text{\rm hom}}+\bbC_\text{\rm hom}\nabla\bu)+\bb=\bzero\qquad &\text{in}\ \Omega, \\[10pt]
\bu = \bu_\circ\qquad &\text{on}\ \partial\Omega,
\end{array}
\right.
\eeqn
where for $x\in\Omega$,
\begin{align}
\label{Srhom}\bS_{r\text{\rm hom}}(x)&=\dashint_{Y}\Big(\hat\bS_r(\bK_x,\bar y)+\hat\bbS(\bbC(\bar y),\bK_x)[\bH^{-\trans}_x\nabla_{\bar y}\hat\bw^\bzero(x,\bar y)\bH^{-1}_x]\Big)\, d\bar y\\
\label{Chom}\bbC_\text{\rm hom}(x)\bE&=\dashint_{Y}\hat\bbS(\bbC(\bar y),\bK_x)\big[\bE+\bH^{-\trans}_x\nabla_{\bar y}\bw^\bE(x,\bar y)\bH^{-1}_x\big]\, d\bar y\qquad \text{for all}\ \bE\in\Sym
\end{align}
and $\hat\bw^\bzero(x,\cdot)\in \cW_\text{\rm per}(Y,\Real^n)$ is the unique solution of 
\beqn\label{ucp0}
\displaystyle\text{\rm div}_{\bar y}\Big\{\bH^{-1}_x\Big(\hat\bS_r(\bK_x,\cdot)+\hat\bbS(\bbC(\cdot),\bK_x)\big[\bH_x^{-\trans}\nabla_{\bar y}\hat\bw^\bzero(x,\cdot)\bH^{-1}_x\big]\Big)\bH^{-\trans}_x\Big\}=\bzero\quad \text{in}\ Y 
\eeqn
and $\bw^\bE(x,\cdot)\in \cW_\text{\rm per}(Y,\Real^n)$ is the unique solution of
\beqn\label{ucpE}
\displaystyle\text{\rm div}_{\bar y}\Big(\bH^{-1}_x\hat\bbS(\bbC(\cdot),\bK_x)\big[\bE+\bH_x^{-\trans}\nabla_{\bar y}\bw^\bE(x,\cdot)\bH^{-1}_x\big]\bH^{-\trans}_x\Big)=\bzero\quad \text{in}\ Y. 
\eeqn
\end{theorem}

\begin{proof}
From Proposition \ref{euprop}, $\bu^\ve$ is bounded in the $H^1$-norm independent of $\ve$.  It follows from Theorem \ref{ltscompact} that there is a $\bu\in H^1(\Omega,\Real^n)$ and $\hat\bu\in L^2(\Omega, \cW_\text{per}(Y_\bH,\Real^n))$ such that along a subsequence (still denoted by $\bu^\ve$)
\beqn
\nabla\bu^\ve\rightarrow \nabla \bu+\nabla_y\hat\bu\qquad \text{\rm locally periodic two-scale as}\ \ \ve\rightarrow 0.
\eeqn
Notice that, using the notation \eqref{Hnotation}, $\bS_r^\ve=\cL^\ve_0(\bS_{r\bH})$ and $\bbC^\ve=\cL^\ve_0(\bbC_\bH)$ where
\beqn
\bS_r(x,\bar y)=\hat\bS_r(\bK_x,\bar y)\quad \text{and}\quad \bbC(x,\bar y)=\hat\bbS(\bbC(\bar y),\bK_x)\rfa (x,\bar y)\in \Omega\times Y.
\eeqn
Multiply \eqref{dispbvp} by the test function $\bv=\bv^0+\ve\cL^\ve_\rho(\psi\bv^1_\bH)$, where $\bv^0\in C^\infty_0(\Omega,\Real^n)$, $\bv^1\in C^\infty_\text{per}(Y,\Real^n)$, and $\psi\in C^\infty_0(\Omega,\Real)$, integrate over $\Omega$ and integrate by parts to obtain
\beqn\label{balveor}
\int_\Omega\Big[\cL^\ve_0(\bS_{r\bH})+\cL^\ve_0(\bbC_\bH)\nabla\bu^\ve\Big]\cdot\Big[\nabla\bv^0+\ve\nabla\big(\cL_\rho^\ve(\psi\bv^1_\bH)\big)\Big]\dv + \int_\Omega\bb\cdot\Big(\bv^0+\ve\cL^\ve_\rho(\psi\bv^1_\bH)\Big)\dv=0.
\eeqn
We want to compute the limit of \eqref{balveor} as $\ve$ goes to zero.  This will be done by considering each term individually.  Begin by expanding the first term in \eqref{balveor} to obtain
\begin{multline}\label{balveexp}
\int_\Omega \cL_0^\ve(\bS_{r\bH})\cdot \nabla\bv^0\dv+\ve\int_\Omega \cL_0^\ve(\bS_{r\bH})\cdot \nabla\big(\cL_\rho^\ve(\psi\bv^1_\bH)\big)\dv\\
+\int_\Omega \cL^\ve_0(\bbC_\bH)\nabla\bu^\ve\cdot \nabla\bv^0\dv+\ve\int_\Omega \cL^\ve_0(\bbC_\bH)\nabla\bu^\ve\cdot \nabla\big(\cL_\rho^\ve(\psi\bv^1_\bH)\big)\dv.
\end{multline}

Looking at the first term in \eqref{balveexp} and using Proposition~\ref{obtsc} we have
\begin{align}
\int_\Omega \cL_0^\ve(\bS_{r\bH})\cdot \nabla\bv^0\dv&\rightarrow \int_\Omega\dashint_{Y_x} \bS_{r\bH}(x,y)\cdot \nabla\bv^0(x)\, dydx
\end{align}
as $\ve$ goes to zero.

Using Assumption 5(c), Corollary \ref{corconv}, and the fact that $\bv^0$ is independent of $\ve$, one finds that
\begin{multline}
\int_\Omega\cL_0^\ve(\bbC_\bH)\nabla\bu^\ve\cdot\nabla\bv^0\dv=\int_\Omega\nabla\bu^\ve\cdot\cL_0^\ve(\bbC_\bH)\nabla \bv^0\dv\\
\rightarrow \int_\Omega\dashint_{Y_x}\big(\nabla\bu(x)+\nabla_y\hat\bu(x,y)\big)\cdot\bbC_\bH(x,y)\nabla \bv^0(x)\, dydx\\
=\int_\Omega\dashint_{Y_x}\bbC_\bH(x,y)\big(\nabla\bu(x)+\nabla_y\hat\bu(x,y)\big)\cdot\nabla\bv^0(x)\, dydx
\end{multline}
as $\ve$ goes to zero.

Before the third term in \eqref{balveexp} is considered, notice that by the chain and product rules
\begin{multline}
\ve\nabla(\cL^\ve_\rho(\psi\bv^1_\bH))(x)=\sum_{k\in I^\ve}\Big[ \ve\bv^1(\bH^{-1}_{x^\ve_k}(x-\tilde x_k^\ve)/\ve)\otimes \nabla \psi(x) \phi_{\Omega_k^\ve}(x)\\
+\psi(x)\nabla \bv^1(\bH^{-1}_{x^\ve_k}(x-\tilde x_k^\ve)/\ve)\bH^{-1}_{x_k^\ve}\phi_{\Omega_k^\ve}(x)+\ve \psi(x)\bv^1(\bH^{-1}_{x_k^\ve}(x-\tilde x_k^\ve)/\ve)\otimes \nabla \phi_{\Omega_k^\ve}(x)\Big].
\end{multline}
Defining
\begin{align}
\bA_1^\ve(x)&=\sum_{k\in I^\ve} \bv^1(\bH^{-1}_{x^\ve_k}(x-\tilde x_k^\ve)/\ve)\otimes \nabla \psi(x) \phi_{\Omega_k^\ve}(x),\\
\bA_2^\ve(x)&=\sum_{k\in I^\ve} \psi(x)\nabla \bv^1(\bH^{-1}_{x^\ve_k}(x-\tilde x_k^\ve)/\ve)\bH^{-1}_{x_k^\ve}\phi_{\Omega_k^\ve}(x),\\
\bA_3^\ve(x)&=\sum_{k\in I^\ve} \psi(x)\bv^1(\bH^{-1}_{x_k^\ve}(x-\tilde x_k^\ve)/\ve)\otimes \nabla \phi_{\Omega_k^\ve}(x),
\end{align}
for $x\in \Omega$, we can write
\begin{multline}\label{badterm}
\ve\int_\Omega\cL^\ve_0(\bS_{r\bH})\cdot\nabla(\cL_\rho^\ve(\psi_\bH\bv^1))\dv= \ve\int_\Omega\cL^\ve_0(\bS_{r\bH})\cdot \bA_1^\ve \dv\\ + \int_\Omega\cL^\ve_0(\bS_{r\bH})\cdot \bA_2^\ve \dv+\ve\int_\Omega\cL^\ve_0(\bS_{r\bH})\cdot \bA_3^\ve \dv.
\end{multline}
Since $\bv^1$ and $\psi$ are bounded in $L^\infty$ and $\phi_{\Omega_k^\ve}$ satisfies \eqref{phi1} and \eqref{phi2}, $\bA_1^\ve$ is bounded in $L^2$ independent of $\ve$ and $\bA_3^\ve$ is bounded in $L^2$ by a constant times $\ve^{-\rho}$.  Thus, since $\cL_0^\ve(\bS_{r\bH})$ is bounded in $L^\infty$ independent of $\ve$ and $\rho<1$, the first and third terms on the right-hand side of \eqref{badterm} go to zero as $\ve$ goes to zero.  For the middle term on the right-hand side of \eqref{badterm}, first notice that by the chain rule $\nabla\bv^1(\bH^{-1}_{x^\ve_k}(x-\tilde x_k^\ve)/\ve)\bH^{-1}_{x^\ve_k}=\nabla_y \bv^1_\bH(x^\ve_k,(x-\tilde x_k^\ve)/\ve)$.  Using this fact and the definition of $\cL_0^\ve$, see \eqref{cLve0}, $\bA_2^\ve$ can be written as
\begin{align}
\bA_2^\ve(x) &= \sum_{k\in I^\ve}\psi(x)\nabla_y\bv^1_\bH(x_k^\ve,(x-\tilde x_k^\ve)/\ve)\psi_{\Omega^\ve_k}(x)\\
\label{badterm2}&=\cL^\ve_0(\nabla_y\bv^1_\bH\psi)(x)+\sum_{k\in I^\ve}\nabla_y\bv_\bH^1(x_k^\ve,(x-\tilde x_k^\ve)/\ve)\Big[\psi(x)\phi_{\Omega_k^\ve}(x)-\psi(x^\ve_k)\chi_{\Omega_k^\ve}(x)\Big].
\end{align}
The second term in \eqref{badterm2} goes to zero in $L^2$ as $\ve$ goes to zero by \eqref{phi1} and since $\psi$ is smooth.  Thus, by Proposition \ref{clconv} we obtain
\begin{align}
\lim_{\ve\rightarrow 0}\int_\Omega\cL^\ve_0(\bS_{r\bH})\cdot \bA_2^\ve \dv&=\lim_{\ve\rightarrow 0}\int_\Omega\cL^\ve_0(\bS_{r\bH})\cdot\cL^\ve_0(\nabla_y\bv^1_\bH\psi)\dv\\
&= \int_\Omega\dashint_{Y_x}\bS_{r\bH}(x,y)\cdot\nabla_y\bv^1_\bH(x,y)\psi(x)\,dydx.
\end{align}

Computing the limit of the fourth term in \eqref{balveexp} is similar to computing the limit of the third term.  Thus, from the arguments in the previous paragraph and Corollary \ref{corconv} we find that
\begin{align}
\lim_{\ve\rightarrow 0}\ve \int_\Omega\cL^\ve_0(\bbC_\bH)\nabla\bu^\ve\cdot\nabla(\cL^\ve_\rho(\psi\bv^1_\bH))\dv&=\lim_{\ve\rightarrow 0}\int_\Omega\cL^\ve_0(\bbC_\bH)\nabla\bu^\ve\cdot \bA_2^\ve \dv\\
&\hspace{-1.2in}=\lim_{\ve\rightarrow 0}\int_\Omega\cL^\ve_0(\bbC_\bH)\nabla\bu^\ve\cdot\cL^\ve_0(\nabla_y\bv^1_\bH\psi)\dv\\
&\hspace{-1.2in}=\lim_{\ve\rightarrow 0} \int_\Omega\nabla\bu^\ve\cdot\cL^\ve_0(\bbC_\bH\nabla_y\bv^1_\bH\psi)\dv\\
&\hspace{-1.2in}= \int_\Omega\dashint_{Y_x}\bbC_\bH(x,y)\big[\nabla\bu(x)+\nabla_y\hat\bu(x,y)\big]\cdot\nabla_y\bv^1_\bH(x,y)\psi(x)\,dydx.
\end{align}

Finally, to deal with the last term on the left-hand side of \eqref{balveor} we use the fact that $\cL^\ve_\rho(\psi_\bH\bv^1)$ is bounded in $L^2$ independent of $\ve$ to find that
\beqn
\ve\int_\Omega\bb\cdot \cL^\ve_\rho(\psi\bv^1_\bH)\dv \rightarrow 0\quad \text{as}\quad \ve\rightarrow 0.
\eeqn	

Putting the results of the last five paragraphs together we find that the limit of \eqref{balveor} as $\ve$ goes to zero is
\begin{multline}\label{ballimit0}
\int_\Omega\dashint_{Y_x}\Big(\bS_{r\bH}(x,y)+\bbC_\bH(x,y)[\nabla\bu(x)+\nabla_y\hat\bu(x,y)]\Big)\cdot\big(\nabla\bv^0(x)+\psi(x)\nabla_y\bv^1_\bH(x,y)\big)\, dydx\\
+\int_\Omega\bb(x)\cdot\bv^0(x)\, dx=0.
\end{multline}
Using the change of variables $y=\bH_x\bar y$ results in
\begin{multline}\label{ballimit1}
\int_\Omega\dashint_{Y}\Big(\bS_{r\bH}(x,\bar y)+\bbC(x,\bar y)\big[\nabla\bu(x)+\nabla_{\bar y}\hat\bu(x,\bH_x\bar y)\bH^{-1}_x\big]\Big)\cdot[\nabla\bv^0(x)+\psi(x)\nabla\bv^1(\bar y)\bH^{-1}_x]\, d\bar ydx\\
+\int_\Omega\bb(x)\cdot\bv^0(x)\, dx=0.
\end{multline}

Consider the case where $\bv^0=\textbf{0}$ and use the fact that $\psi\in C^\infty_0(\Omega,\Real)$ was arbitrary to conclude from \eqref{ballimit1} that
\beqn\label{ballimit2}
\dashint_{Y}\Big(\bS_{r\bH}(x,\bar y)+\bbC(x,\bar y)\big[\nabla\bu(x)+\nabla_{\bar y}\hat\bu(x,\bH_x\bar y)\bH^{-1}_x\big]\Big)\cdot\nabla\bv^1(\bar y)\bH^{-1}_x\, d\bar y=0\quad \text{for a.e.}\ x\in\Omega.
\eeqn
Fix $x\in\Omega$.  For each $\bv^1\in C^\infty_\text{per}(Y,\Real^n)$ we can define $\overline\bv\in C^\infty_\text{per}(Y,\Real^n)$ by $\overline\bv=\bH^{-\trans}_x\bv^1$.  Thus, \eqref{ballimit2} can be written as 
\begin{equation}\label{ballimit3}
\dashint_{Y}\bH^{-1}_x\Big(\bS_r(x,\bar y)+\bbC(x,\bar y) \big[\nabla\bu(x)+\nabla_{\bar y}\hat\bu(x,\bH_x\bar y)\bH^{-1}_x\big]\Big)\bH^{-\trans}_x\cdot\nabla\overline\bv(\bar y)\, d\bar y=0.
\end{equation}
Given $\bE\in\Lin$, let $\hat\bw^\bE(x,\cdot)\in \cW_\text{per}(Y,\Real^n)$ be the unique solution of
\beqn\label{ucp}
\displaystyle\text{div}_{\bar y}\Big\{\bH^{-1}_x\Big(\bS_r(x,\cdot)+\bbC(x,\cdot)\big[\bE+\bH_x^{-\trans}\nabla_{\bar y}\hat\bw^\bE(x,\cdot)\bH^{-1}_x\big]\Big)\bH^{-\trans}_x\Big\}=\bzero\quad \text{in}\ Y. 
\eeqn
The existance of a unique solution to this unit cell problem, as well as all of the other unit cell problems such as \eqref{ucp0} and \eqref{ucpE}, follow from the Assumptions and the Lax--Milgram theorem.  Notice that $x$ acts as a parameter in these equations.  It follows that $\hat \bu$ in \eqref{ballimit3} is of the form
\beqn\label{hatu}
\hat\bu(x,\bH_x\bar y)=\bH^{-\trans}_x\hat\bw^{\nabla\bu(x)}(x,\bar y)\qquad \text{for a.e.}\ x\in\Omega,\ \bar y\in Y.
\eeqn
Moreover, because of the structure of \eqref{ucp}, $\hat\bw^\bE$ is an affine function of $\bE$.  In particular, $\hat\bw^\bE$ is of the form 
\beqn
\hat\bw^\bE=\hat\bw^\bzero+\bw^\bE
\eeqn
where $\hat\bw^\bzero(x,\cdot)\in \cW_\text{per}(Y,\Real^n)$ is the unique solution of \eqref{ucp0}, which is \eqref{ucp} with $\bE=\bzero$, and $\bw^\bE(x,\cdot)\in \cW_\text{per}(Y,\Real^n)$ is the unique solution of \eqref{ucpE}.  The function $\bw^\bE(x,\cdot)$ depends linearly on the symmetric part of $\bE$.

Going back to \eqref{ballimit1}, now consider the case where $\psi=0$ and use \eqref{hatu} to obtain
\begin{multline}\label{ballimit4}
\int_\Omega\dashint_{Y}\Big(\bS_r(x,\bar y)+\bbC(x,\bar y)\big[\nabla\bu(x)+\bH^{-\trans}_x\nabla_{\bar y}\hat\bw^\bE(x,\bar y)\bH^{-1}_x\big]\Big)\cdot\nabla\bv^0(x)\, d\bar ydx\\
+\int_\Omega\bb(x)\cdot\bv^0(x)\, dx=0.
\end{multline}
Motivated by this equation, define $\bS_{r\text{hom}}$ by \eqref{Srhom} and $\bbC_\text{hom}$ by \eqref{Chom} so that \eqref{ballimit4} can be written as
\beqn
\int_\Omega(\bS_{r\text{hom}}+\bbC_\text{hom}\nabla\bu)\cdot\nabla\bv^0\dv+\int_\Omega\bb\cdot\bv^0\dv=\bzero.
\eeqn
Since this holds for all $\bv^0\in C^\infty_0(\Omega,\Real^n)$, $\bu$ is the weak solution of \eqref{dispbvphom}.

Using the same arguments as in the classical periodic microstructure case, see for example Cioranescu and Donato \cite{CD} or Oleinik, Shomeav and Yosifian \cite{OSY}, it can be shown that the effective elasticity tensor \eqref{Chom} satisfies the properties mentioned in Assumption 5(a)--(c).  It then follows from the Lax--Milgram theorem that  \eqref{dispbvphom} has a unique solution and, hence, the entire sequence $\bu^\ve$ converges to $\bu$, not just a subsequence.
\end{proof}

Locally periodic microstructure can be used to approximate nonperiodic microstructures.  To see how this cane be done, consider $\bL,\bM:\overline\Omega\rightarrow\Lin$ with $\bL$ of class $C^2$ and $\bM$ of class $C^1$ such that $\bL_x$ is invertible at each $x\in\Omega$.  Moreover, assume that the linear mapping
\beqn\label{Hlinear}
\bu\mapsto \bL^{-1}_{x}\bu+\nabla(\bL^{-1})_{x}[\bu]x
\eeqn
is invertible.  The functions $\bL$ and $\bM$ can be used to defined an elastic material with nonperiodic microstructure.  Consider the elastic material with residual stress $\bS_{r\text{np}}^\ve$ and elasticity tensor $\bbC^\ve_\text{np}$ defined at $x\in\Omega$ by
\begin{align}
\bS_{r\text{np}}^\ve(x)&=\hat\bS_r(\bM_x,\bL^{-1}_xx/\ve)\\
\bbC^\ve_\text{np}(x)&=\hat\bbS(\bbC(\bL^{-1}_xx/\ve),\bM_x).
\end{align}
To construct the locally periodic approximation, begin by fixing $r\in(1/2,1)$ and for each $k\in I^\ve$ choose $x_k\in\Omega_k^\ve$ arbitrarily.  It follows from a Taylor series expansion about $x_k^\ve$ that for all $x\in\Omega_k^\ve$
\begin{align}
\bL^{-1}_x x/\ve&=\bL^{-1}_{x_k^\ve}x_k^\ve/\ve+\bL^{-1}_{x_k^\ve}(x-x_k^\ve)/\ve+\nabla(\bL^{-1})_{x_k^\ve}[x-x_k^\ve]x_k^\ve/\ve+O(\ve^{2r-1}).
\end{align}
Define $\bH_x$ to be the inverse of the linear mapping in \eqref{Hlinear} so that
\beqn\label{keyapprox}
\bL^{-1}_x x/\ve=\bL^{-1}_{x_k^\ve}x_k^\ve/\ve+\bH^{-1}_{x_k^\ve}(x-x_k^\ve)/\ve+O(\ve^{2r-1}).
\eeqn
Next, by the periodicity of $\hat\bS_r$ and $\bbC$, we can find $\tilde x_k^\ve\in\Omega_k^\ve$ such that all $x\in\Omega_k^\ve$
\begin{align}
\hat\bS_r(\bM_x,\bL^{-1}_{x_k^\ve}x_k^\ve/\ve+\bH^{-1}_{x_k^\ve}(x-x_k^\ve)/\ve)&=\hat\bS_r(\bM_x,\bH^{-1}_{x_k^\ve}(x-\tilde x_k^\ve)/\ve)\\
\bbC(\bL^{-1}_{x_k^\ve}x_k^\ve/\ve+\bH^{-1}_{x_k^\ve}(x-x_k^\ve)/\ve)&=\bbC(\bH^{-1}_{x_k^\ve}(x-\tilde x_k^\ve)/\ve).
\end{align}

Motivated by the above discussion, consider the locally periodic microstructure defined by
\begin{align}
\label{YStrans1}\bS_r^\ve(x)&=\sum_{k\in I^\ve} \hat\bS_r(\bK_{x^\ve_k},\bH_{x^\ve_k}^{-1}(x-\tilde x_k^\ve)/\ve)\chi_{\Omega_k^\ve}(x),\\
\label{YCtrans1}\bbC^\ve(x)&=\sum_{k\in I^\ve} \hat\bbS(\bbC(\bH_{x^\ve_k}^{-1}(x-\tilde x_k^\ve)/\ve),\bK_{x^\ve_k})\chi_{\Omega_k^\ve}(x),
\end{align}
where $\bK=\bM$ and $\bH$ is the inverse of the linear mapping defined in \eqref{Hlinear}.  Assuming $1/2<r<1$ and using \eqref{keyapprox}, \eqref{YStrans1}, and \eqref{YCtrans1} it can be shown that
\begin{align}
\|\bS^\ve_r-\bS^\ve_{r\text{np}}\|_{L^2(\Omega)} + \|\bbC^\ve-\bbC^\ve_{\text{np}}\|_{L^2(\Omega)}&\rightarrow 0\quad \text{as}\ \ve\rightarrow 0.
\end{align}
This fact allows for the arguments in Theorem~\ref{thmmain} to be used to show that the solutions $\bu^\ve$ of the system
\beqn\label{dispbvp1}
\left \{
\begin{array}{ll}
\text{div}(\bS_{r\text{np}}^\ve+\bbC_\text{np}^\ve\nabla\bu^\ve)+\bb=\textbf{0}\qquad &\text{in}\ \Omega, \\[10pt]
\bu^\ve = \bu_\circ\qquad &\text{on}\ \partial\Omega,
\end{array}
\right.
\eeqn
converge to the function $\bu$ that satisfies \eqref{dispbvphom} with $\bS_{r\text{hom}}$ and $\bbC_{\text{home}}$ given by \eqref{Srhom} and \eqref{Chom}, respectively.

Notice that when $\bM=\bL$ in the nonperiodic structure, the locally periodic approximation does \emph{not} satisfy $\bK=\bH$.

\section{Interpretation of results and special cases}\label{sectInterp}


Several comments about Theorem~\ref{thmmain} are warranted.  We see from \eqref{Chom} that the effective elasticity tensor $\bbC_\text{hom}$ it is not constant.  This is contrary to the classical case when the microstructure is periodic.  Since the effective elasticity tensor $\bbC_\text{hom}$ depend on $x\in\Omega$, it may initially appear that the unit cell problem \eqref{ucpE} must be solved at each $x$.  From a theoretical standpoint, this is not a problem.  However, to calculate the effective elasticity tensor numerically, solving \eqref{ucpE} at each $x$ is tantamount to solving a system of linear equations at each point in a mesh.  Clearly this is infeasible.  A more sensible strategy is to compute the effective elasticity tensor at a few points in the domain and use these values to define a piecewise constant elasticity tensor that approximates \eqref{Chom}.   This kind of approach was studied by Shkoller \cite{Shk}.  However, this method is problematic if $\bH$ and $\bK$ have large gradients.  Since the solution of \eqref{ucpE}, and hence $\bbC_\text{hom}$, only depend on $x$ through $\bK_x$ and $\bH_x$, another strategy is available.  Rather than solving \eqref{ucpE} for different values of $x$, it can be solved for different values of $\bK$ and $\bH$ and the resulting effective elasticity tensor can be calculated using \eqref{Chom}.  To obtain the effective elasticity tensor for values of $\bK$ and $\bH$ not explicitly considered, an interpolation can be used.  This same method can be applied to the effective residual stress since the solution of \eqref{ucp0} only depends on $x$ through $\bK_x$ and $\bH_x$ as well.  

Interestingly, the effective residual stress $\bS_{r\text{hom}}(x)$ given in \eqref{Srhom} is not solely given by the average of $\hat\bS_r(\bK_x,\cdot)$ over the unit cell $Y$, as one might conjecture.  Rather, it is given by this average plus an additional term involving the elasticity tensor $\bbC$ and a supplemental function $\hat\bw^\bzero$ that is found by solving the unit cell problem \eqref{ucp0}.   Motivated by the form of $\hat\bS_r$ given in \eqref{Ssp} and the fact that $Y$ was a stress-free configuration, let us suppose $\hat\bS_r(\bQ,\cdot)=\bzero$ when $\bQ$ is orthogonal.  In this case, if $\bK_x$ is orthogonal it follows from the structure of \eqref{ucp0} that $\hat\bw^\bzero(x,\cdot)=\bzero$ and, hence, $\bS_{r\text{hom}}(x)=\bzero$.  Put more succinctly, when $\bK_x$ is orthogonal, the effective residual stress vanishes at $x$.  There is another case in which the effective residual stress is irrelevant.  When $\bH$ and $\bK$ are constant the microstructure is periodic.  In this case, the effective residual stress is independent of $x$ and, hence, constant.  Thus, the divergence of $\bS_{r\text{hom}}$ is zero so it plays no role in the macroscopic equation \eqref{dispbvphom}.  This means that the consideration of residual stresses is only interesting when nonperiodic microstructures are considered, not in the classical periodic case.


Next consider when $\bS^\ve_r$ and $\bbC^\ve$ are specified by \eqref{Ssp} and \eqref{bbSsp}, so that the formula for the macroscopic elasticity tensor $\bbC_\text{hom}$ \eqref{Chom} and the unit cell problem \eqref{ucpE} are given by
\beqn\label{homCsp}
\bbC_\text{hom}(x)\bE=\dashint_{Y}\bK_x\bbC(\bar y)\big[\bK_x^\trans\bE\bK_x+\bK_x^\trans\bH^{-\trans}_x\nabla_{\bar y}\bw^\bE(x,\bar y)\bH^{-1}_x\bK_x\big]\bK_x^\trans\, d\bar y
\eeqn
and
\beqn\label{ucp0sp}
\displaystyle\text{div}_{\bar y}\Big(\bH^{-1}_x\bK_x\bbC\big[\bK_x^\trans\bE\bK_x+\bH_x^{-\trans}\bK_x^\trans\nabla_{\bar y}\bw^\bE(x,\cdot)\bK_x\bH^{-1}_x\big]\bK_x^\trans\bH^{-\trans}_x\Big)=\bzero\quad \text{in}\ Y. 
\eeqn
The formula for the effective residual stress and the unit cell problem \eqref{ucp0} for these choices of $\bS^\ve_r$ and $\bbC^\ve$ are similar and, hence, are not explicitly written.  See Figure~\ref{figcases}(a) for an example of a nonperiodic microstructure that can be approximated by a locally periodic microstructure in which $\bH\not=\bK$.  Several special cases are of particular interest.  

\begin{figure}
\centering
\includegraphics[height=2in]{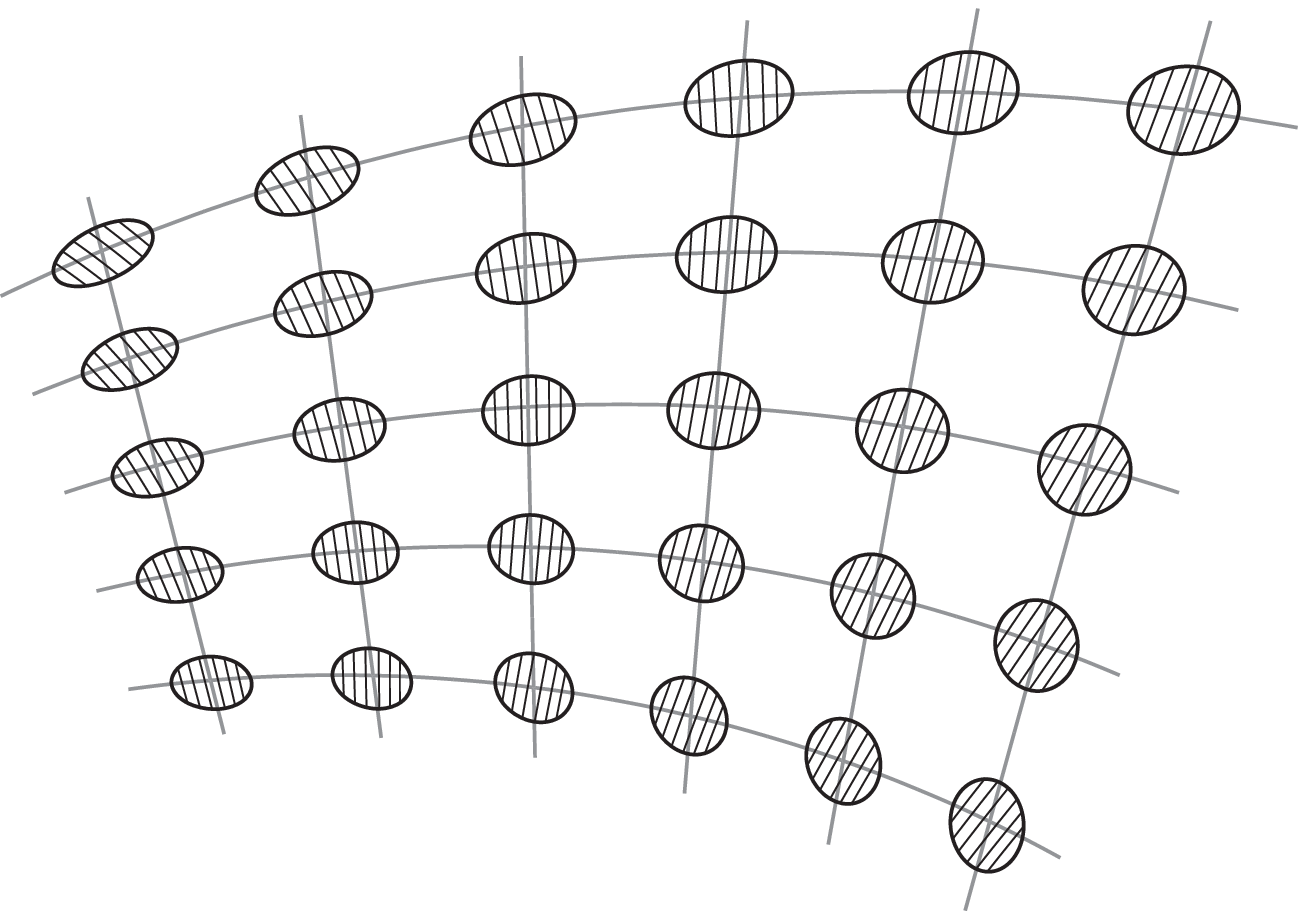}
\vspace{.1in}
\includegraphics[height=2in]{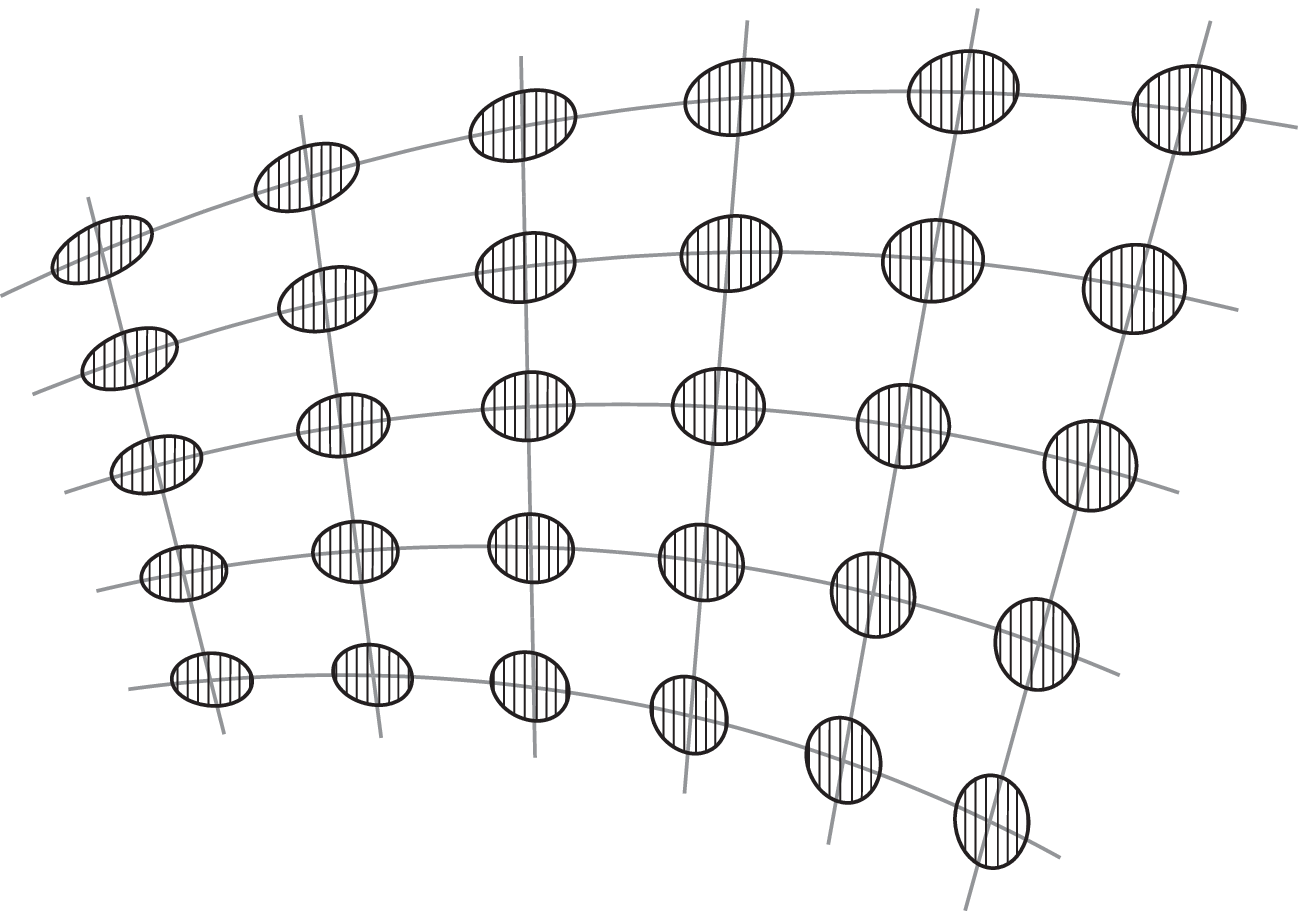}
\includegraphics[height=2in]{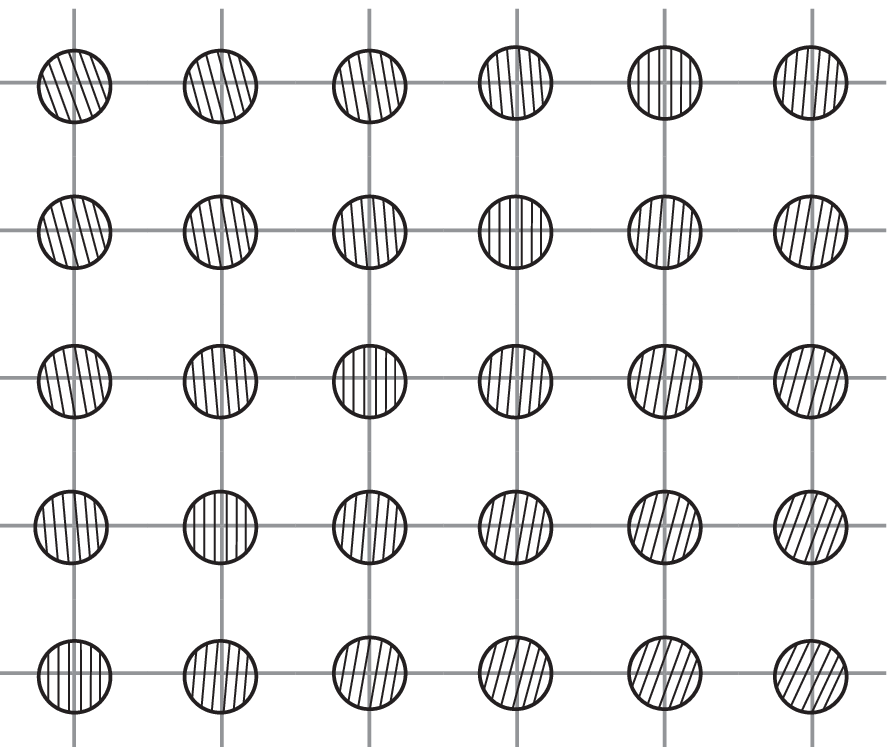}
\hspace{.3in}
\includegraphics[height=2in]{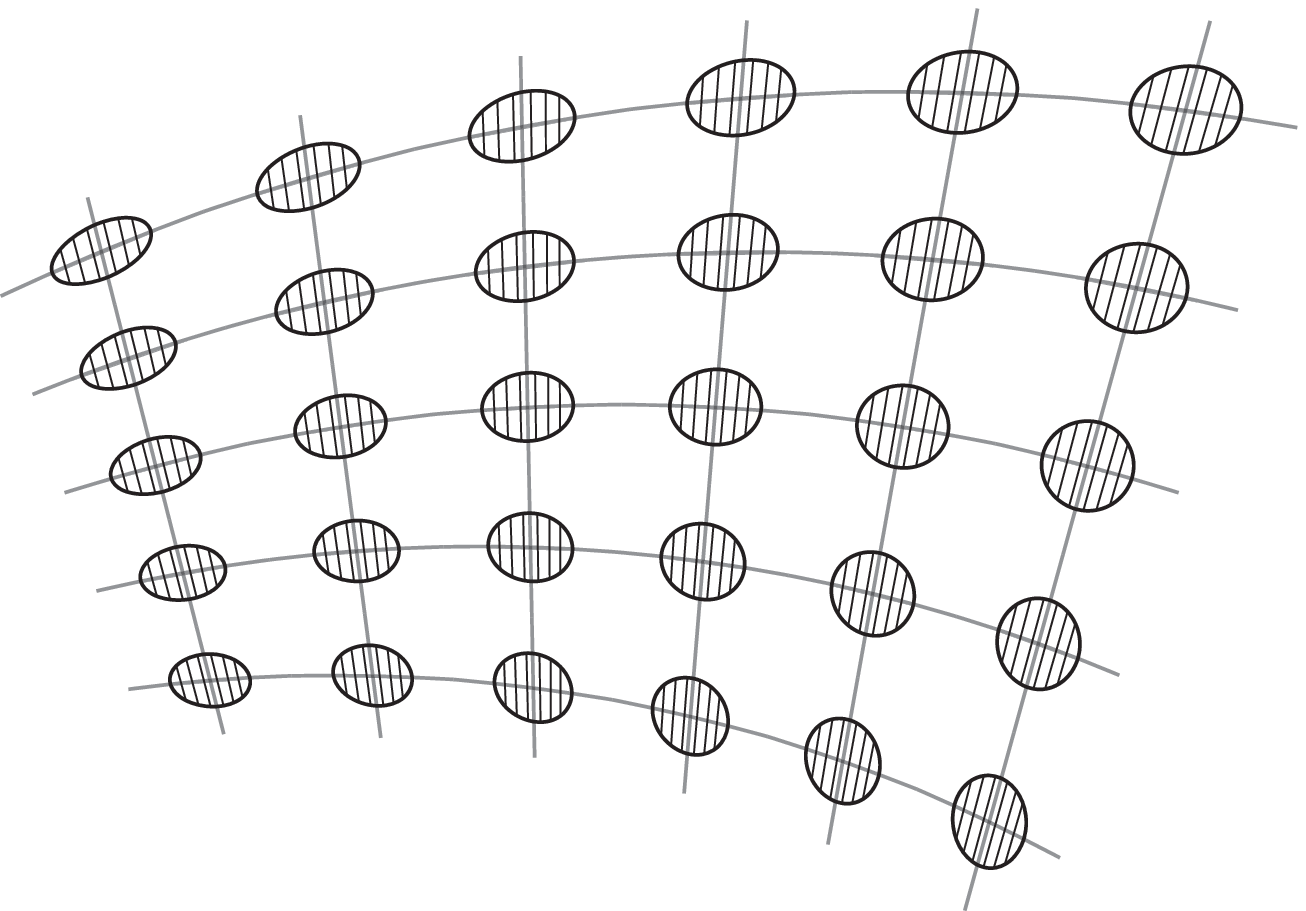}
\thicklines
\put(-330,170){(a)}
\put(-115,170){(b)}
\put(-115,-15){(d)}
\put(-330,-15){(c)}
\caption{A depiction of several microstructures that can be approximated by locally periodic microstructures for different choices of $\bH$ and $\bK$.  These microstructures consist of two materials: an isotropic matrix and anisotropic spheres.  The parallel lines in the spheres represent the orientation of the anisotropy within the spheres.  The grid lines are present to help visualize how far the microstructure deviates from being periodic.  (a) This microstructure can be approximated by a locally periodic microstructure in which $\bH$ and $\bK$ are not equal.  Notice that the alignment between the parallel lines within each sphere and the reference lines differs from sphere to sphere. (b) This microstructure depicts the $\bK=\bone$ case and is similar to the kinds of microstructure considered by Ptashnyk \cite{Plpts} and Briane \cite{Briane94}.  Notice that the alignment of the lines within the spheres is always vertical, indicating that the anisotropy of the spheres is not changing.  (c) This microstructure is approximated by a locally periodic microstructure with $\bH=\bone$.  For this kind of microstructure, the alignment of the anisotropy changes but the arrangement of the spheres is periodic. (d)  Here the case $\bH=\bK$ is depicted.  The the parallel lines within each sphere are aligned with the reference lines in every sphere.}
\label{figcases}
\end{figure}

\subsection{$\bK=\textbf{1}$ }

In this case \eqref{homCsp} and \eqref{ucp0sp} become
\beqn\label{homC1}
\bbC_\text{hom}(x)\bE=\dashint_{Y}\bbC(\bar y)\big[\bE+\bH^{-\trans}_x\nabla_{\bar y}\bw^\bE(x,\bar y)\bH^{-1}_x\big]\, d\bar y
\eeqn
and
\beqn\label{ucp1}
\displaystyle\text{div}_{\bar y}\Big(\bH^{-1}_x\bbC\big[\bE+\bH_x^{-\trans}\nabla_{\bar y}\bw^\bE(x,\cdot)\bH^{-1}_x\big]\bH^{-\trans}_x\Big) =\bone\quad \text{in}\ Y.
\eeqn
Since $\bone$ is orthogonal, the effective residual stress $\bS_r$ vanishes.  Since $\bK$ describes how the alignment of the anisotropy of $\bbC$ change from patch to patch, in the case $\bK=\bone$, the alignment of the anisotropy does not change.  See Figure~\ref{figcases}(b).

For a particular choice of $\bH$ this case was considered by Ptashnyk \cite{Plpts}.  The above analysis could also of been carried out for the Poisson equation rather than the equation of elasticity.  An analysis of a microstructure similar to this was carried out by Briane \cite{Briane94} when $\bH$ is given by the gradient of a function.  

\subsection{$\bH=\textbf{1}$}

Under this assumption \eqref{homCsp} and \eqref{ucp0sp} become
\beqn\label{homC2}
\bbC_\text{hom}(x)\bE=\dashint_{Y}\bK_x\bbC(\bar y)\big[\bK_x^\trans\bE\bK_x+\bK_x^\trans\nabla_{\bar y}\bw^\bE(x,\bar y)\bK_x\big]\bK_x^\trans\, d\bar y
\eeqn
and
\beqn\label{ucp2}
\displaystyle\text{div}_{\bar y}\Big(\bK_x\bbC\big[\bK_x^\trans\bE\bK_x+\bK_x^\trans\nabla_{\bar y}\bw^\bE(x,\cdot)\bK_x\big]\bK_x^\trans\Big)=\bone\quad \text{in}\ Y. 
\eeqn
Here $\bbC_\text{hom}$ gives the elasticity tensor of a material whose elastic properties on the microscale are being transformed by $\bK$ but the orientation of the unit cell remains unchanged throughout the microstructure.  See Figure~\ref{figcases}(c).


\subsection{$\bH=\bK$}

Perhaps the most interesting case is when $\bH=\bK$.  See Figure~\ref{figcases}(d) for a depiction of a nonperiodic microstructure that can be approximated by a locally periodic microstructure corresponding to $\bH=\bK$.  Under this assumption, \eqref{ucp0sp} becomes
\beqn\label{ucp3}
\displaystyle\text{div}_{\bar y}\Big(\bbC\big[\bK_x^\trans\bE\bK_x+\nabla_{\bar y}\bw^\bE(x,\cdot)\big]\Big)=\bzero\quad \text{in}\ Y. 
\eeqn
If, for any $\bE\in\Lin$, $\tilde\bw^\bE\in \cW_\text{per}(Y,\Real^n)$ is the unique solution of 
\beqn\label{ucp3.1}
\displaystyle\text{div}_{\bar y}\Big(\bbC\big[\bE+\nabla_{\bar y}\tilde\bw^\bE\big]\Big)=\textbf{0}\quad \text{in}\ Y, 
\eeqn
then the solution of \eqref{ucp3} must have the form
\beqn
\bw^\bE(x,\bar y)=\tilde\bw^{\bK_x^\trans\bE\bK_x}(\bar y)\quad \text{for a.e.}\ x\in\Omega,\ \bar y\in Y.
\eeqn
Using $\tilde\bw$, the effective elasticity tensor \eqref{homCsp} can be written as
\beqn\label{homC3}
\bbC_\text{hom}(x)\bE=\dashint_{Y}\bK_x\bbC(\bar y)\big[\bK_x^\trans\bE\bK_x+\nabla_{\bar y}\tilde\bw^{\bK_x^\trans\bE\bK_x}(\bar y)\big]\bK_x^\trans\, d\bar y.
\eeqn
This macroscopic elasticity tensor is obtained by solving the unit cell problem \eqref{ucp3.1} which does not depend on $x\in\Omega$.  Moreover, notice that for all $x_1,x_2\in\Omega$, if we set $\bM(x_1,x_2)=\bK_{x_1}^{-1}\bK_{x_2}$, then
\beqn\label{Cmu}
\bbC_\text{hom}(x_2)\bE=\bM(x_1,x_2)\bbC_\text{hom}(x_1)\big[\bM(x_1,x_2)^\top\bE\bM(x_1,x_2)\big]\bM(x_1,x_2)^\trans\rfa \bE\in\Lin.
\eeqn
Thus, knowing the macroscopic elasticity tensor at any one point determines the elasticity tensor at any other point.  Once the unit cell problem \eqref{ucp3.1} is solved and the average in \eqref{homC3} is computed at one point, \eqref{Cmu} can be used to find the effective elasticity tensor at all other points.  Hence, unlike the other previously considered cases where theoretically an infinite number of unit cell problems have to be solved to determine the effective elasticity tensor, here the unit cell problem \eqref{homC3} only has to be solved six times, once for each element of a basis for $\Sym$, to completely determine $\bbC_\text{hom}$.

The situation involving the residual stress is not the same.  As has already been mentioned, the solution of \eqref{ucp0} only depends on $x$ through the values $\bH_x$ and $\bK_x$.  Motivated by this, let $\hat\bw^0(\bK,\cdot)\in\cW_\text{per}(Y,\Real^n)$ be the unique solution of 
\beqn\label{ucp03}
\displaystyle\text{\rm div}_{\bar y}\big(\hat\bS(\bK^\trans\bK,\cdot)+\bbC\nabla_{\bar y}\hat\bw^\bzero(\bK,\cdot)\big)=\bzero 
\eeqn
so that the effective residual stress is given by
\beqn\label{Srhom3}
\bS_{r\text{\rm hom}}(x)=\dashint_{Y}\bK_x\Big(\hat\bS(\bK^\trans_x\bK_x,\bar y)+\bbC(\bar y)\nabla_{\bar y}\hat\bw^\bzero(\bK_x,\bar y)\Big)\bK_x^\trans\, d\bar y.
\eeqn
It is not possible to relate the residual stress at different points in a way that is analogous for $\bbC_\text{hom}$ in \eqref{Cmu}, and so \eqref{ucp03} must be solved for different values of $\bK$ to obtain $\bS_{r\text{\rm hom}}(x)$ for different values of $x$.  If we set
\beqn
\hat\bT_\text{hom}(\bE,x)=\hat\bS_\text{hom}(x)+\bbC_\text{hom}(x)\bE,
\eeqn
then from \eqref{Cmu} and \eqref{Srhom3} it follows that $\hat\bT_\text{hom}$ satisfies
\beqn
\hat\bT_\text{hom}(\bE,x_1)=\bM(x_1,x_2)\hat\bT_\text{hom}\big(\bM(x_1,x_2)^\top\bE\bM(x_1,x_2),x_2\big)\bM(x_1,x_2)^\top
\eeqn
for all $x_1,x_2\in\Omega$.  This says that in the homogenized limit, the material is materially uniform with $\bM$ being the material uniformity and $\bK$ being a uniform reference.  Moreover, if $\bK$ is the gradient of a function, then the material is also homogeneous.  However, in general, the material is inhomogeneous.  For a detailed discussion of material uniformity and inhomogeneity see, for example, Epstein and El$\dot{\text{z}}$anowski \cite{EE}, Noll \cite{N67}, or Wang \cite{W67}.

As mentioned in the introduction, microstructures appearing in nature often can be approximated by locally periodic microstructures in which $\bH=\bK$.  Such materials have anisotropic components that make up their microstructure.  As the orientation of the microstructure changes from point to point, the orientation of the anisotropies changes as well.  A prime example of such a microstructure occurs in plant cell walls in which the anisotropic fibrils change their orientation throughout the thickness of the cell wall.

\bibliographystyle{acm}
\bibliography{locper} 

\end{document}